\documentclass[12pt,oneside,english]{amsart}
\usepackage[latin1]{inputenc}
\usepackage{amssymb}

\makeatletter

\newtheorem{lemma}{Lemma}

\usepackage{babel}

\makeatother
\begin{document}
\baselineskip=17pt
\title[On Erd\H{o}s constant]{On Erd\H{o}s constant}

\author{Vladimir Shevelev}
\address{Department of Mathematics \\Ben-Gurion University of the
 Negev\\Beer-Sheva 84105, Israel. e-mail:shevelev@bgu.ac.il}


\begin{abstract}
In 1944, P. Erd\H{o}s \cite{1} proved that if
$n$ is a large highly composite number (HCN) and $n_1$ is the next HCN, then
$$n<n_1<n+n(\log n)^{-c},$$
where $c>0$ is a constant. In this paper, using numerical results by D. A. Corneth,
we show that most likely $c<1.$
\end{abstract}

\maketitle

\section{Introduction}
In 1915, Ramanujan \cite{2} introduced \slshape highly composite numbers
\upshape \enskip (HCN) as positive integers $n$ such that $d(m)<d(n)$ for all $m<n.$
(cf. A002182\cite{3}). At the same time, he introduced a more wide sequence of
\slshape largely composite numbers\upshape \enskip (LCN's) as
positive integers $n$ such that $d(m)\leq d(n)$ for all $m<n.$ (cf. A067128\cite{3})
In 1944, P. Erd\H{o}s \cite{1}, strengthening inequality of Ramanujan,
proved that if $n$ is a large HCN
and $n_1$ is the next HCN, then
\begin{equation}\label{1}
 n<n_1<n+n(\log n)^{-c},
\end{equation}
  where $c>0$ is a constant. This result is equivalent to the following: the number
of HCN's$<=x$ is greater than $(\log x)^{1+c}.$ At the beginning of the article he
writes: "At present I cannot decide whether the number of HCN's not exceeding $x$
is greater than $(\log x)^{\kappa}$ for every $\kappa."$  In this paper, using numerical results by D. A. Cormeth, we show that most likely $c<1,$ or, the same, in the cited Erd\H{o}s'
question, only $\kappa<2.$

\section{Sequence A273379 [3]}
Erdos \cite{1} noted that every HCN is divisible by every prime less than its
greatest prime divisor $p.$ The author with P. J. C. Moses considered
the sequence "LCN's $n$ which are not divisible by all the primes $< p,$ where $p$
is the greatest prime divisor of $n."$ The first few numbers of this sequence are
\begin{equation}\label{2}
3, 10, 20, 84, 168, 336, 504, 660, 672, 3960, 4680, 32760, 42840, ...
\end{equation}
Consider prime power factorization of these terms $\geq10:$
$$2\cdot5,\enskip 2^2\cdot5,\enskip 2^2\cdot3\cdot7,\enskip
 2^3\cdot3\cdot7,\enskip 2^4\cdot3\cdot7,\enskip
 2^3\cdot3^2\cdot7,\enskip 2^2\cdot3\cdot5\cdot11,\enskip 2^5\cdot3\cdot7,$$
$$2^3\cdot3^2\cdot5\cdot11,\enskip 2^3\cdot3^2\cdot5\cdot13, \enskip2^3\cdot3^2\cdot5\cdot7\cdot13,\enskip 2^3\cdot3^2\cdot5\cdot7\cdot17,...$$
Note that among the first 12 terms $\geq10,$ in ten terms there is missed only one
prime between the greatest prime divisor and the second greatest prime divisor,
while in two terms there missed two primes. D. A. Corneth (A273415\cite{3}) found the
smallest LCN's with
$i$ missed primes between the greatest prime divisor and the second greatest prime
divisor, such that smaller primes all divide the terms, $i=1,2,...,10:$
$$10, 4680, 6585701522400, 193394747145600, 27377180785991836800,$$
 $$29378941900252048776672000,
 5384823686347760468943298225056000,$$
 $$404593694258692410380118300618528000,$$
 $$1714431214566179268370439406441900195214656000,$$
\begin{equation}\label{3}
 180656647480221782329653424360823828484237888000.
\end{equation}
He asks, whether this sequence is infinite?
\section{Infiniteness of sequence (3)}
Let $n$ be a large HCN. Let its greatest prime divisor be the $k$-th prime number
$p_k,\enskip k=k(n).$ It is known \cite{1}, that $p_k||n.$
\begin{lemma}\label{L1}\cite{1}
\begin{equation}\label{4}
c_1\log n <p_k<c_2\log n.
\end{equation}
\end{lemma}
By Lemma \ref{L1} and prime number theorem, we have
\begin{equation}\label{5}
k=O(\log n/\log\log n).
\end{equation}
\begin{lemma}\label{L2}
Let $n_1$ be the next HCN after $n.$ Then all numbers in the interval  $[n, n_1)$ 
of the form
\begin{equation}\label{6}
np_{k+1}/p_k,...,np_{k+r}/p_k,
\end{equation}
if they exist, are LCN's.
\end{lemma}
\begin{proof} All numbers (\ref{6}) have the same number of divisors as $n,$ and
between them there is no any HCN, since the smallest $HCN>n$ is $n_1$ which is larger
every number (\ref{6}).
\end{proof}
\begin{lemma}\label{L3}
Let $N$ be a term of sequence $(\ref{3})$ with $r$ missed primes between the greatest
prime divisor $p_{k+r}$ and the second greatest prime
divisor $p_{k-1}$ of $N.$
Then together with $N$ all numbers
$$Np_{k+r-1}/p_{k+r}, Np_{k+r-2}/p_{k+r},...,Np_{k}/p_{k+r}$$
are LCN's (but not HCN's, except for the last number).
\end{lemma}
\newpage
\begin{proof}
Indeed, in the opposite case, between these numbers there is a HCN $\leq N,$ but
since all they, including $N,$ have the same number of divisors, it would contradict
the condition, that $N$ is LCN. 
\end{proof}
Lemma \ref{L3} means that every number of sequence (\ref{3}) is building from an
HCN (which always has not any gap between its prime divisors $2,...,p_k$) by 
consecutive multiplication
by $p_{k+1}/p_k, p_{k+2}/p_{k+1},..., p_{k+r}/p_{k+r-1}$ with the possible maximal
$r.$ By (\ref{1}) and (\ref{6}), 
$$n<np_{k+r}/p_k<n(1+(\log n)^{-c}).$$

In order to have a real chance to obtain the numbers (\ref{3}), let require
a stronger inequality
\begin{equation}\label{7}
p_{k+r}/p_k<1+\frac{1}{2}(\log n)^{-c}.
\end{equation}
Here $\frac{1}{2}$ could be changed by any smaller positive constant.
Note that (\ref{7}) means that
\begin{equation}\label{8}
(1+o(1))(k+r)\log(k+r)/(k\log k)<1+\frac{1}{2}(\log n)^{-c},
\end{equation}
 where, by (\ref{5}), $k=O(\log n/\log\log n).$ Since, for $r<k,$
  $\log(k+r)=\log k +\log(1+r/k)=\log k+r/k+O((r/k)^2),$
  then $\log(k+r)/\log k=1+O(r/(k\log k).$
 So, by (\ref{8}), we see that (\ref{7}) yields 
 $$1+r/k+O(r/(k\log k))<1+\frac{1}{2}(\log n)^{-c},$$
 or
 $$r\leq \frac{k}{2(\log n)^c}=O((\log n)^{1-c}/\log\log n).$$ Thus if and only
  if $c<1,$ the value of $r$ could be arbitrary large for sufficiently large n.
Moreover, the existence the numbers (\ref{3}) allows to conjecture that, indeed,
 $0<c<1.$

\end{document}